\documentclass[12pt]{article}
\usepackage[cp1251]{inputenc}
\usepackage[russian]{babel}
\usepackage{amsmath,amssymb,amsthm}
\newtheorem{theorem}{Theorem}
\newtheorem{lemma}{Lemma}

\theoremstyle{definition}

\renewenvironment{proof}{
\par\noindent{\it Proof.}} { \mbox{}\hfill $\blacksquare$ \par }

\DeclareMathOperator{\per}{per} \DeclareMathOperator{\SSym}{Sym}
\DeclareMathOperator{\sgn}{sign} \DeclareMathOperator{\dif}{d}
\DeclareMathOperator{\charr}{char} \DeclareMathOperator{\Tr}{Tr}
\DeclareMathOperator{\rank}{rank} \DeclareMathOperator{\Gal}{Gal}
\DeclareMathOperator*{\Coef}{Coef}
\renewcommand{\mod}{\mathop{{\rm mod}}}
\newcommand{\pd}{\partial}
\newcommand{\lam}{\lambda}
\newcommand{\sumn}{\sum\limits_{i=1}^n}
\newcommand{\vep}{\varepsilon}

\newcommand{\ls}{\left(}
\newcommand{\rs}{\right)}

\newcommand{\lk}{\left[}
\newcommand{\rk}{\right]}

\date{}
\begin{document}

\begin{center}
\Large{Determinants and permanents of an arbitrary \\
the Hadamard degree of a Cauchy matrix and proof of a generalization
of a conjecture of R.F.Scott(1881).}
\end{center}
\begin{center}
A.M.Kamenetsky
\end{center}
\vspace{2em}
\begin{center}
Abstract.
\end{center}
{\small In this paper we give the absolutely new proof of a
conjecture of R.F.Scott(1881) on the permanent of a Cauchy matrix
$\ls \frac{1}{x_i-y_j} \rs_{1 \leqslant i,j \leqslant n},$ where
$x_1, \ldots, x_n$ and $y_1, \ldots, y_n$~are the distinct roots of
the polynomials $x^n-1$ and $y^n +1,$ respectively. The simple
formula is given for the permanent of the Cauchy matrix  $A= \ls
\frac{1}{x_i-y_j} \rs_{1 \leqslant i,j \leqslant n},$ where $x_1,
\ldots, x_n$ and  $y_1, \ldots, y_n$~are the distinct roots of the
polynomials $x^n+a$ and $y^n +b$, respectively:
\begin{gather*}
\per (A) =\frac{n}{(b-a)^n}
\prod_{k=1}^{n-1}[nb-k(b-a)] =\\
=\begin{cases}
(-1)^{\frac{n-1}{2}} \cfrac{n}{(b-a)^n}
\prod\limits_{k=1}^{\frac{n-1}{2}}[-na-k(b-a)][nb-k(b-a)],
\mbox{if $n \equiv 1 ( \mod 2)$,} \\
\cfrac{n}{2} \cdot \cfrac{n(a+b)}{(b-a)^n}
\prod\limits_{k=1}^{\frac{n}{2}-1}[na+k(b-a)][nb+k(a-b)], \mbox{if $n\equiv 0(\mod 2)$}.
\end{cases}
\end{gather*}
from which the corrected formula of R.F.Scott follows instantly.
Proof follows from obtained by the author a formula for the
determinant of an arbitrary of the Hadamard degree $m$ of a Cauchy
matrix $A$ and Borchard's theorem.

}

\vspace{2em}
\begin{center}
{\large 1.~Introduction.}
\end{center}
In 1881 R.F.Scott [1] gave the following result without proof. Let
$$
A=((x_i-y_j)^{-1})_{1\leqslant i ,j \leqslant n }
$$
 be a Cauchy matrix, where
$x_1,\ldots,x_n $ and $y_1, \ldots, y_n$~are the distinct roots of the polynomials
$x^n-1 \mbox{ and } y^n+1$, respectively. Then
\begin{equation}
\per(A)=\begin{cases} n[1\times 3\times 5 \times \ldots \times(n-2)]^2/2^n,
&\text{if $n$ is odd,}\\
0, &\text{if $n$ is even.}
\end{cases}
\end{equation}

In 1979 Henryc Minc [2] first proved the correct result
\begin{equation}
 \per(A)=\begin{cases}
(-1)^{(n-1)/2}n[1\times 3\times 5 \times \ldots \times(n-2)]^2/2^n,&\text{if $n$ is odd,}\\
0, &\text{if $n$ is even.}
\end{cases}
\end{equation}

Thus R.F.Scott guess formula for $\per(A)$ up to a sign. In [3] the proof of the
 formula (2) is the same as in [2]. Both in [3] and in [2] evaluate of the
determinant of the first and the second Hadamard degree foregoing
specialized a Cauchy matrix which is based in fact on the proof of
one and the same trigonometric identity. My proof is based on
absolutely another idea, permitting obtain compact formula for
determinant of an arbitrary Hadamard degree of specialized a Cauchy
matrix if $x_1, x_2, \ldots x_n $ and $y_1, y_2, \ldots y_n $~are
the distinct roots of polynomials $x^n+a \mbox{ and } y^n+b$,
respectively.
\begin{center}
{\large 2.~Results.}
\end{center}
\begin{theorem}
Let $x_1, x_2, \ldots, x_n$~be independent variables over a
field~$K$ of characteristic~0 and let $y_1,y_2, \ldots, y_n \in K$,
$A=((x_i-y_j)^{-1})_{1 \leqslant i, j\leqslant n}$~be a~Cauchy
matrix, let $A^{(m)}=((x_i-y_j)^{-m})_{i \leqslant i, j \leqslant
n}$~be $m$-th Hadamard degree of the Cauchy matrix~$A$,~let $m
\geqslant 1$. Then
\begin{gather}
\det (A^{(m)})=(-1)^{n(m-1)}[(m-1)!]^{-n} \ls \frac{\pd}{\pd x_1}\rs^{m-1}
\ls \frac{\pd}{\pd x_2} \rs^{m-1} \ldots \ls \frac{\pd}{\pd x_n} \rs^{m-1} \det (A), \\
\per (A^{(m)})=(-1)^{n(m-1)}[(m-1)!]^{-n} \ls \frac{\pd}{\pd x_1} \rs ^{m-1}
\ls\frac{\pd}{\pd x_2}\rs^{m-1} \ldots \ls \frac{\pd}{\pd x_n} \rs^{m-1} \per (A).
\end{gather}
\end{theorem}

\begin{proof}
Since
\begin{multline*}
\per(A^{(m)})=\sum_{\sigma \in \SSym (n)} \prod_{i=1}^n (x_i-y_{\sigma(i)})^{-m} \mbox{
and} \det(A^{(m)})=\sum_{\sigma \in \SSym (n)} (\sgn \sigma) \prod_{i=1}^n
(x_i-y_{\sigma(i)})^{-m},
\end{multline*}
then
\begin{multline*}
\ls \frac{\pd}{\pd x_1} \rs ^{m-1} \ls \frac{\pd}{\pd x_2} \rs ^{m-1} \ldots \ls
\frac{\pd}{\pd
x_n}\rs ^{m-1} \per (A) =\\
=\ls \frac{\pd}{\pd x_1} \rs^{m-1} \ls \frac{\pd}{\pd x_2} \rs ^{m-1} \ldots \ls
\frac{\pd}{\pd x_n} \rs^{m-1} \sum_{\sigma \in
\SSym (n)} \prod_{i=1}^n (x_i-y_{\sigma(i)})^{-1}= \\
=\sum_{\sigma \in \SSym (n)} \prod_{i=1}^n \ls \frac{\pd}{\pd x_i}
\rs^{m-1}(x_i-y_{\sigma(i)})^{-1}
=\sum_{\sigma \in \SSym (n)} \prod_{i=1}^n(-1)^{m-1}(m-1)!(x_i-y_{\sigma(i)})^{-m}=\\
=(-1)^{n(m-1)} [(m-1)!]^n \sum_{\sigma \in \SSym (n)} \prod_{i=1}^n
(x_i-y_{\sigma(i)})^{-m}= (-1)^{n(m-1)} [(m-1)!]^n \per(A^{(m)}),
\end{multline*}
\begin{multline*}
\ls \frac{\pd}{\pd x_1}\rs^{m-1} \ls\frac{\pd}{\pd x_2}\rs^{m-1} \ldots \ls\frac{\pd}{\pd
x_n}\rs^{m-1} \det (A) = \\
=\ls\frac{\pd}{\pd x_1}\rs^{m-1} \ls\frac{\pd}{\pd x_2}\rs^{m-1}
\ldots \ls\frac{\pd}{\pd x_n}\rs^{m-1} \sum_{\sigma \in \SSym (n)} (\sgn \sigma) \prod_{i=1}^n (x_i-y_{\sigma(i)})^{-1}= \\
=\sum_{\sigma \in \SSym (n)} (\sgn \sigma) \prod_{i=1}^n
\ls\frac{\pd}{\pd x_i}\rs^{m-1} (x_i-y_{\sigma(i)})^{-1}=\\
= \sum_{\sigma \in \SSym (n)} (\sgn \sigma) \prod_{i=1}^n (-1)^{m-1}(m-1)!(x_i-y_{\sigma(i)})^{-m}=\\
=(-1)^{n(m-1)} [(m-1)!]^n \sum_{\sigma \in \SSym (n)} (\sgn \sigma) \prod_{i=1}^n
(x_i-y_{\sigma(i)})^{-m} =(-1)^{n(m-1)} [(m-1)!]^n \det(A^{(m)}).
\end{multline*}

\end{proof}
\begin{lemma}
Let $K$~be a field and $f(x) \in K[x],$ let $E$~be the splitting
field of $f(x)$ over $K$.
$$
f(x)=\prod_{i=1}^n (x-x_i), \quad g(x)=x^n f(x^{-1}) =\prod_{i=1}^n (1-x_i x), \;
x_i \in E, \; 1 \leqslant i \leqslant n
$$
$ p_k=p_k(x_1,x_2,\ldots,x_n) = \sum_{i=1}^n x_i^k$. Then
\begin{equation}
\sum_{k=1}^{\infty}p_k x^k=-x \ls \frac{\dif}{\dif x} g(x) \rs
(g(x))^{-1}
\end{equation}
\end{lemma}

\begin{proof}
\begin{multline*}
\sum_{k=1}^{\infty} p_k x^k = \sum_{i=1}^n \sum_{k=1}^{\infty}
x_i^k x^k =
\sum_{i=1}^n (1-x_i x)^{-1}x_i x = -x \sumn \frac{-x_i}{1-x_i x} = \\
=-x \sum_{i=1}^n \frac{\frac{\dif}{\dif x} (1-x_i x)}{1-x_i x} = -x \ls \frac{\dif}{\dif
x} g(x) \rs (g(x))^{-1}
\end{multline*}

\end{proof}
\begin{lemma}
Let $K$~be a field of characteristic 0, let $F(x) \in K[[x]]$, let
$m, k$~be positive integers. Then
\begin{equation}
\frac{m}{k} \Coef_{x^m}(F(x))^k = \Coef_{x^{m-1}}((F(x))^{k-1}
\frac {\dif}{\dif x} F(x)).
\end{equation}
If $F(0) = 0$, then
\begin{equation}
\sum_{k=1}^m \frac{m}{k} \Coef_{x^m}(F(x))^k = \Coef_{x^m}-
\frac{x \frac{\dif}{\dif x} (1 -  F(x))}{1-F(x)}.
\end{equation}
\end{lemma}
\begin{proof}
If $F(x)  = \sum\limits_{n=0}^{\infty} a_n x^n$, then by definition
$\Coef\limits_{x^m} F(x) = a_m$. Therefore $\Coef\limits_{x^m} F(x)
= \frac{1}{m} \times \break \times \Coef\limits_{x^{m-1}}\frac{\dif
}{\dif x} F(x) $ and it follows that
$$
\Coef\limits_{x^m}(F(x))^k \!=\! \frac{1}{m} \Coef\limits_{x^{m-1}}
\frac{\dif }{\dif x} (F(x))^k\! = \!\frac{1}{m}
\Coef\limits_{x^{m-1}} k (F(x))^{k-1} \frac{\dif}{\dif x } F(x) \!=
\frac{k}{m} \Coef\limits_{x^{m-1}}((F(x))^{k-1} \frac{\dif}{\dif x}
F(x)),
$$
i.~d. $\frac{m}{k} \Coef\limits_{x^m} (F(x))^k =
\Coef\limits_{x^{m-1}} ((F(x))^{k-1} \frac{\dif}{\dif x} F(x))$. Let
$F(0) =0$. Then by the equality (6) it follows that
\begin{multline*}
\sum_{k=1}^m \frac{m}{k} \Coef_{x^m} (F(x))^k = \sum_{k=1}^m
\Coef_{x^{m-1}} \Bigl( (F(x))^{k-1} \frac{\dif}{\dif x} F(x) \Bigr) =
\Coef_{x^{m-1}} \Bigl(\frac{\dif}{\dif x} F(x)\Bigr) \sum_{k=1}^{m}
(F(x))^{k-1} = \\ =\Coef_{x^{m-1}} \Bigl(\frac{\dif}{\dif x} F(x) \Bigr)
\sum_{k=1}^{\infty} (F(x))^{k-1} =
\Coef_{x^{m-1}}\Bigl( \frac{\dif}{\dif x} F(x) \Bigr)(1 - F(x))^{-1} = \\ =
\Coef_{x^m} x \Bigl( \frac{\dif}{\dif x} F(x) \Bigr) (1-F(x))^{-1} =
\Coef_{x^m}\Bigl[ -x \Bigl(\frac{\dif}{\dif x} (1-F(x))\Bigr) (1 -F(x))^{-1} \Bigr]
\end{multline*}
Lemma 2 is proved.
\end{proof}

If $\charr K = 0,$ $\varphi(x) \in K[[x]]$, $\varphi(0)=1$, then by definition
$\log(\varphi (x)) = - \sum\limits_{m=1}^{\infty}
\frac{(1-\varphi(x))^m}{m}$. Therefore
$$
\frac{\dif}{\dif x} \log(\varphi(x)) = -\biggl(\sum\limits_{m=1}^{\infty}
(1 - \varphi(x))^{m-1}\biggr) \frac{\dif}{\dif x} (1 -\varphi(x)) =
(\varphi(x))^{-1}\frac{\dif}{\dif x} \varphi(x).
$$
Therefore the equatlity (7) follows from the equality
$$
m \Coef_{x^m} \log
(1-F(x)) = \Coef_{x^m}\Bigl(x \frac{\dif}{\dif x} \log(1-F(x))\Bigr)
$$

\begin{lemma}
Let $K$~be a field of characteristic 0, let $\varphi_1 (x),\varphi_2
(x), \ldots, \varphi_n (x) \!\in \!K[[x]],$ let $\varphi_i (0)=1,$
$1 \leqslant i \leqslant n.$ Then
\begin{equation}
\sum_{i=1}^n \log (\varphi_i (x))=\log \ls \prod_{i=1}^n \varphi_i (x) \rs
\end{equation}
\end{lemma}
\begin{proof}
Since
\begin{multline*}
\frac{\dif}{\dif x} \sumn \log(\varphi_i (x)) = \sumn \frac{\dif}{\dif x}
\log(\varphi_i (x))= \sumn (\varphi_i (x))^{-1} \frac{\dif}{\dif x} \varphi_i (x),\\
\frac{\dif}{\dif x} \log \ls \prod_{i=1}^n \varphi_i (x) \rs= \ls
\prod_{i=1}^n \varphi_i (x)\rs^{-1}
\frac{\dif}{\dif x} \ls \prod_{i=1}^n \varphi_i (x) \rs =\\
=\biggl(\prod_{i=1}^n \varphi_i (x)\biggr)^{-1} \sum_{j=1}^n \ls \prod_{i=1}^n \varphi_i (x) \rs
(\varphi_j (x))^{-1} \frac{\dif}{\dif x} \varphi_j (x)= \sum_{j=1}^n(\varphi_j (x))^{-1}
\frac{\dif}{\dif x} \varphi_j (x)
\end{multline*}
and constant terms of the series $\sumn \log (\varphi_i (x))$ and
$\log\ls \prod\limits_{i=1}^n \varphi_i (x) \rs$ are equal to $0$,
then
$$
 \sumn \log(\varphi_i (x)) = \log\ls \prod_{i=1}^n \varphi_i (x)\rs.
$$
\end{proof}

\begin{lemma}
Let $n \geqslant 1$, let $l_1,l_2, \ldots, l_n$~be  positive
integers such that $l_1 \geqslant l_2 \geqslant \ldots \geqslant
l_n$. Then
\begin{equation}
\frac{l_1+l_2+ \ldots + l_n}{l_1}
\prod_{i=1}^{n-1}\binom{l_i}{l_{i+1}} = \biggl(\prod_{i=1}^{n-1}
\binom{l_i}{l_{i+1}} \biggr) + \sum_{k=1}^{n-1} \biggl(\prod_{i=1}^k
\binom{l_i-1}{l_{i+1}-1} \biggr) \biggl(\prod_{i=k+1}^{n-1}
\binom{l_i}{l_{i+1}} \biggr)
\end{equation}
\end{lemma}
{\it Proof} by induction on $n$. Let $n \geqslant 2$. Then from the
equality $\binom{l_1}{l_2} = \frac{l_1}{l_2} \binom{l_1 -1}{l_2-1}$
it follows that $\frac{l_1+l_2+\ldots+l_n}{l_1}
\prod\limits_{i=1}^{n-1} \binom{l_i}{l_{i+1}} =
(1+\frac{l_2+\ldots+l_n}{l_1}) \prod\limits_{i=1}^{n-1}
\binom{l_i}{l_{i+1}} = \Bigl(\prod\limits_{i=1}^{n-1}
\binom{l_i}{l_{i+1}} \Bigr) + \frac{l_2+ \ldots +
l_n}{l_1}\frac{l_1}{l_2} \binom{l_1 -1}{l_2 -1}
\prod\limits_{i=2}^{n-1}\binom{l_i}{l_{i+1}} = \break
=\Bigl(\prod\limits_{i=1}^{n-1} \binom{l_i}{l_{i+1}}\Bigr) +
\binom{l_1 -1}{l_2 -1}\frac{l_2 + \ldots + l_n}{l_2}
\prod\limits_{i=2}^{n-1} \binom{l_i}{l_{i+1}}$. By the induction
hypothesis it follows that $\frac{l_2+\ldots + l_n}{l_2} \times
\break \times \prod\limits_{i=2}^{n-1} \binom{l_i}{l_{i+1}} =
\Bigl(\prod\limits_{i=2}^{n-1} \binom{l_i}{l_{i+1}} \Bigr) +
\sum\limits_{k=2}^{n-1}\Bigl( \prod\limits_{i=2}^k \binom{l_i
-1}{l_{i+1}-1}\Bigr)\Bigl(\prod\limits_{i=k+1}^{n-1}
\binom{l_i}{l_{i+1}}\Bigr).$ Therefore
$\frac{l_1+l_2+ \ldots +l_k}{l_1} \prod\limits_{i=1}^{n-1}
\binom{l_i}{l_{i+1}} = \break = \! \Bigl(\prod\limits_{i=1}^{n-1} \!
\binom{l_i}{l_{i+1}}\Bigr)+ \binom{l_1 -1}{l_2 -1}
\Bigl[\prod\limits_{i=2}^{n-1} \! \binom{l_i}{l_{i+1}} +
\sum\limits_{k=2}^{n-1} \Bigl(\prod\limits_{i=2}^k \!
\binom{l_i-1}{l_{i+1} -1} \Bigr)\Bigl(\prod\limits_{i=k+1}^{n-1} \!
\binom{l_i}{l_{i+1}}\Bigr) \Bigr]\! = \!
\Bigl(\prod\limits_{i=1}^{n-1}\binom{l_i}{l_{i+1}} \Bigr) +
\sum\limits_{k=1}^{n-1} \Bigl(\prod\limits_{i=1}^k \binom{l_i
-1}{l_{i+1}-1}\Bigr) \times \break \times
\Bigl(\prod\limits_{i=k+1}^{n-1} \binom{l_i}{l_{i+1}}\Bigr). $
Lemma 4 is proved.
\begin{lemma}
Let $\lam_1,\lam_2, \ldots, \lam_n$~be nonnegative integers,
$\sum\limits_{i=1}^n i \lam_i = \!m > \!1$. Then $\frac{m((\lam_1 +
\lam_2 + \ldots + \lam_n-1)!)}{\lam_1!\lam_2! \ldots \lam_n!}$ is
integer.
\end{lemma}
\begin{proof}
Lemma 5 follows directly from the lemma 4. Really, let $s\! =
\max\{i \mid 1 \! \leqslant i \!\leqslant n, \, \sum\limits_{j=i}^n
\lam_j \!\ge 1\}$, $l_i = \sum\limits_{j=i}^n \lam_j$, $1 \leqslant
i \leqslant n$. Then $s \geqslant 1$, $l_i - l_{i+1} = \lam_i$, $1
\leqslant i \leqslant n-1$, $l_n = \lam_n$. Since
$\prod\limits_{i=1}^{n-1}\binom{l_i}{l_{i+1}} = \break =
\frac{l_1!}{(l_1 - l_2)!(l_2 - l_3)! \ldots (l_{n-1} - l_n)! l_n!},$
\, $\sum\limits_{i=1}^{n} l_i =
\sum\limits_{i=1}^{n}\sum\limits_{j=i}^{n} \lam_j =
\sum\limits_{j=1}^{n}\sum\limits_{i=1}^j \lam_j =
\sum\limits_{j=1}^{n}j \lam_j = m,$ then $\frac{l_1+ l_2 +\ldots +
l_s}{l_1} \times \break \times \prod\limits_{i=1}^{s-1}
\binom{l_i}{l_{i+1}} = \frac{l_1+ l_2 + \ldots + l_n}{l_1}
\prod\limits_{i=1}^{n-1} \binom{l_i}{l_{i+1}} = \frac{m}{\lam_1 +
\lam_2 + \ldots + \lam_n} \cdot \frac{(\lam_1 + \lam_2 + \ldots
+\lam_n)!}{\lam_1! \lam_2! \ldots \lam_n!} = \frac{m((\lam_1 +
\lam_2 + \ldots +\lam_n-1)!)}{\lam_1! \lam_2! \ldots \lam_n!}$
\end{proof}
Lemma 5 also follows directly from the lemma 2. Let
$k\!=\sum\limits_{i=1}^n \lam_i$, let $a_1, a_2, \ldots, a_n$~be
independent variables over the field
 $\mathbb{Q}$ of rational numbers, let $K = \mathbb{Q} (a_1, a_2,
\ldots, a_n)$ and $F(x) = \sum\limits_{i=1}^n a_i x^i$, let
$\mathbb{Z}$~be the ring of rational integers. Then
\begin{multline}
\frac{m}{k}\Coef_{x^m}(F(x))^k = \frac{m}{k}\Coef_{x^m}(a_1 x + a_2
x^2 + \ldots + a_n x^n)^k = \\ = \frac{m}{k} \! \sum
\limits_{\substack{\mu_1 + \ldots + \mu_n =k \\ \mu_i \geqslant 0,
\, \mu_i \in \mathbb{Z}, \, 1 \leqslant i \leqslant n \\ \mu_1 + 2
\mu_2 + \ldots + n \mu_n =m}} \!\! \frac{k!}{\mu_1! \mu_2! \ldots
\mu_n!} \prod_{i=1}^n a_i^{\mu_i} =
\\ = \sum \limits_{\substack{\mu_1 + \ldots + \mu_n =k \\ \mu_i
\geqslant 0, \, \mu_i \in \mathbb{Z}, \, 1 \leqslant i \leqslant n \\
\mu_1 + 2 \mu_2 + \ldots + n \mu_n =m}} \frac{m((\mu_1 + \mu_2 +
\ldots + \mu_n -1)!)}{\mu_1! \mu_2! \ldots \mu_n!} \prod_{i=1}^n
a_i^{\mu_i}.
\end{multline}

But from equality (6) it follows that
$$
\frac{m}{k}\Coef\limits_{x^m}(F(x))^k =
\Coef\limits_{x^{m-1}}((F(x))^{k-1} \frac{\dif}{\dif x} F(x))  \in
\mathbb{Z}[a_1, a_2, \ldots, a_n].
$$
Therefore
$$
\sum \limits_{\substack{\mu_1 + \ldots + \mu_n =k \\ \mu_i \geqslant
0, \, \mu_i \in \mathbb{Z}, \, 1 \leqslant i \leqslant n  \\ \mu_1 +
2 \mu_2 + \ldots + n \mu_n =m}} \frac{m((\mu_1 + \mu_2 + \ldots +
\mu_n -1)!)}{\mu_1! \mu_2! \ldots \mu_n!} \prod_{i=1}^n a_i^{\mu_i}
\in \mathbb{Z}[a_1, a_2, \ldots, a_n]
$$
Hence, in particular, $ \frac{m((\mu_1 + \mu_2 + \ldots + \mu_n
-1)!)}{\mu_1! \mu_2! \ldots \mu_n!} \in \mathbb{Z}$. Another one
more proof follows from [10,~the\-o\-rem~8,~p.18].

\begin{lemma}
Let $A$ and $B$ be commutative rings, let $x_1,x_2,\ldots,x_n$~be
independent variables over the ring $A$, let $y_1, y_2, \ldots, y_n
\in B$, let $\varphi$~be a homomorphism of the ring $A$ into the
ring $B$. Let
 $\psi: A[x_1, x_2, \ldots, x_n] \to B$ be defined by
$$
\psi\biggl(\sum a_{\lam_1,\lam_2, \ldots, \lam_n } x_1^{\lam_1}
x_2^{\lam_2} \ldots x_n^{\lam_n} \biggr) = \biggl(\sum \varphi(a_{\lam_1,\lam_2,
\ldots, \lam_n }) y_1^{\lam_1} y_2^{\lam_2} \ldots y_n^{\lam_n} \biggr).
$$
Then the mapping $\psi$~is homomorphism of the ring $A[x_1, x_2,
\ldots, x_n]$ into the ring $B$.
\end{lemma}
{\it Proof} by induction on $n$. Let $f(x), g(x)\! \in \! A[x],$
$f(x)\! = \! \sum\limits_{i=0}^k a_i x^i,$ $g(x) \!=
\!\sum\limits_{j=0}^l b_j x^j,$ $y \in B$, $\psi(f(x)) = \break
=\sum\limits_{i=0}^k \varphi(a_i) y^i.$ Then $f(x) + g(x) =
\sum\limits_{i=0}^{\max(k,l)} (a_i + b_i) x^i$, $f(x) g(x) =
\sum\limits_{s=0}^{k+l} c_s x^s,$ $c_s = \sum\limits_{i+j=s, \, i,j
\geqslant 0} a_i b_j$, $\psi(f(x) + g(x)) =
\psi\Bigl(\sum\limits_{i=0}^{\max(k,l)}(a_i + b_i) x^i \Bigr) =
\sum\limits_{i=0}^{\max(k,l)}(\varphi(a_i + b_i)) y^i =
\sum\limits_{i=0}^{\max(k,l)}(\varphi(a_i) + \varphi(b_i)) y^i =
\break = \sum\limits_{i=0}^{k}\varphi(a_i)y^i +
\sum\limits_{i=0}^{l}\varphi(b_i)y^i = \psi(f(x)) + \psi(g(x)),$
\begin{multline*}
\psi(f(x)g(x)) = \sum\limits_{s=0}^{k+l}
\varphi(c_s)y^s = \sum\limits_{s=0}^{k+l}
\biggl(\varphi  \biggl(\sum\limits_{i+j=s}a_i b_j \biggr) \biggr)y^s = \\ =
\sum\limits_{s=0}^{k+l} \biggl( \sum\limits_{i+j=s} \varphi(a_i)\varphi(b_j) \biggr)
y^s = \biggl( \sum_{i=0}^k (\varphi(a_i))y^i \biggr)
\biggl(\sum_{j=0}^l (\varphi(b_i)) y^j \biggr) = \psi (f(x)) \psi(g (x)).
\end{multline*}
 Let $n \geqslant 2,$ let $\psi_1$~be restriction
of the mapping $\psi$ on subring $A[x_1,x_2, \ldots, x_{n-1}]$ of
the ring $A[x_1,x_2, \ldots, x_n]$. By the induction hypothesis for
$n-1$ it follows that $\psi_1$~is the homomorphism of the ring
$A[x_1,\ldots,x_{n-1}]$ into the ring $B$. From the proved above
case $n=1$ it follows that $\psi$~is homomorphism of the ring
$(A[x_1,\ldots,x_{n-1}])[x_n]$ into the ring $B$. But
$(A[x_1,\ldots,x_{n-1}])[x_n] = \break = A[x_1,\ldots,x_n]$ and
therefore $\psi$~is homomorphism of the ring $A[x_1,\ldots,x_n]$
into the ring $B$. Lemma 6 is proved.
\begin{lemma}
Let $K$~be a field, $f(x) \in K[x]$, $E$~be the splitting field of
 $f(x)$ over $K$, $f(x) = x^n + \break +\sum\limits_{i=1}^n a_i x^{n-i} =
\prod\limits_{i=1}^n (x-x_i),$ $x_i \in E$, $1 \leqslant i \leqslant n$. Then for all $m \geqslant 1$
\begin{equation}
\sum_{i=1}^n x_i^m = \sum_{\substack{\lam_1 + 2\lam_2 + n \lam_n =
m, \\ \lam_i \geqslant 0,\, \lam_i \in \mathbb{Z}, \, 1 \leqslant i \leqslant n }}
(-1)^{\lam_1 + \lam_2 + \ldots + \lam_n} \frac{m((\lam_1 + \lam_2
+ \ldots + \lam_n-1)!)}{\lam_1 ! \lam_2 ! \ldots  \lam_n !}
\prod_{i=1}^n a_i^{\lam_i}
\end{equation}
\end{lemma}
\begin{proof}
We shall prove at first the equality (11) under the condition that
$\charr K = 0.$

1) Let $F(x) = -\sum\limits_{i=1}^n a_i x^i$. Then by the formula
(10) it follows that
\begin{equation*}
\sum_{k=1}^m \frac{m}{k} \Coef_{x^m}(F(x))^k =
\sum_{\substack{\lam_1 + 2\lam_2 + n \lam_n = m, \\ \lam_i \ge
0,\, \lam_i \in \mathbb{Z}, \, 1 \leqslant i \leqslant n }} (-1)^{\lam_1 +
\lam_2 + \ldots + \lam_n} \frac{m((\lam_1 + \lam_2 + \ldots +
\lam_n-1)!)}{\lam_1 ! \lam_2 ! \ldots  \lam_n !} \prod_{i=1}^n
a_i^{\lam_i}
\end{equation*}
Since $1-F(x) = 1 + \sum\limits_{i=1}^n a_i x^i = x^n f(x^{-1})$,
then by the lemma 1 it follows that
$$
\Coef_{x^m} - x \biggl(\frac{\dif}{\dif x}
(1-F(x))\biggr) (1-F(x))^{-1} = \sum_{i=1}^n x_i^m.
$$
Therefore from the equality (7) it follows the equality (11).

2) Let $g(x) = x^n f(x^{-1}) = 1 + \sum\limits_{i=1}^n a_i x^i,$
$\varphi(x) = 1 - g(x) = -\sum\limits_{i=1}^n a_i x^i$. Then by the equality
(5) it follows that
$$
\sum_{m=1}^{\infty} P_m x^{m-1} = \biggl(\frac{\dif }{\dif x} \varphi(x)\biggr)(1
-\varphi(x))^{-1} = \biggl(\frac{\dif}{\dif x} \varphi(x) \biggr)
\sum\limits_{l=0}^{\infty} (\varphi(x))^l.
$$
Hence and by the formula of differentiation of formal series it follows that
\begin{multline*}
\frac{\dif}{\dif x} \sum_{m=1}^{\infty} \frac{P_m}{m} x^m =
\sum_{m=1}^{\infty} P_m x^{m-1} = \biggl(\frac{\dif}{\dif x} \varphi(x) \biggr)
\sum_{l=0}^{\infty} (\varphi(x))^l = \\
=\frac{\dif}{\dif x} \sum_{l=0}^{\infty} \frac{1}{l+1}
(\varphi(x))^{l+1} = \frac{\dif}{\dif
x}\sum_{k=1}^{\infty}\frac{1}{k}(\varphi(x))^k
\end{multline*}
and since constant terms of the series $\sum\limits_{m=1}^{\infty}
\frac{P_m}{m} x^m$ and $\sum\limits_{k=1}^{\infty} \frac{1}{k}
(\varphi(x))^k$ are equal to 0, then
\begin{equation}
\sum\limits_{m=1}^{\infty} \frac{P_m}{m} x^m =
\sum\limits_{m=1}^{\infty} \frac{1}{k} (\varphi(x))^k.
\end{equation}
From the equality (12) it follows that
\begin{multline*}
\frac{P_m}{m} = \sum_{k=1}^m \frac{1}{k} \Coef_{x^m} (\varphi(x))^k =
\sum_{k=1}^m \frac{1}{k} \sum_{\substack{\lam_1+\lam_2+ \ldots +
\lam_n = k\\ \lam_1 + 2\lam_2 + n \lam_n = m, \\ \lam_i \in
\mathbb{Z}, \, \lam_i \geqslant 0,\, 1 \leqslant i \leqslant n }} \frac{k!}{\lam_1!
\lam_2! \ldots \lam_n!} \prod_{i=1}^n (-a_i)^{\lam_i} = \\ =
\sum_{\substack{\lam_1 + 2\lam_2 + n \lam_n = m, \\ \lam_i \in
\mathbb{Z}, \, \lam_i \geqslant 0,\, 1 \leqslant i \leqslant n }}
(-1)^{\lam_1+\lam_2+ \ldots + \lam_n} \frac{(\lam_1+\lam_2+ \ldots
+ \lam_n-1)!}{\lam_1! \lam_2! \ldots \lam_n!} \prod_{i=1}^n
(a_i)^{\lam_i}
\end{multline*}
and therefore
$$
P_m=\sum_{\substack{\lam_1 + 2\lam_2 + n \lam_n = m, \\ \lam_i \in
\mathbb{Z}, \, \lam_i \geqslant 0,\, 1 \leqslant i \leqslant n }}
(-1)^{\lam_1+\lam_2+ \ldots +\lam_n} \frac{m(\lam_1+\lam_2+ \ldots
+ \lam_n-1)!}{\lam_1! \lam_2! \ldots \lam_n!} \prod_{i=1}^n
(a_i)^{\lam_i}
$$

3) The equality (12) follows directly from the equality (8). Really,
$-\sum\limits_{k=1}^{\infty} \frac{1}{k} (\varphi(x))^k = \break = \log(1 - \varphi(x))
= \log(g(x)) = \log\Bigl(1 + \sum\limits_{i=1}^n a_i x^i\Bigr) =
\log\Bigl(\prod\limits_{i=1}^n(1-x_i x)\Bigr) =
\sum\limits_{i=1}^n \log(1-x_i x) = \break= -
\sum\limits_{i=1}^n \sum\limits_{m=1}^{\infty} \frac{(x_i x)^m}{m} = -
\sum\limits_{m=1}^{\infty} \sum\limits_{i=1}^n \frac{(x_i x)^m}{m} =
-\sum\limits_{m=1}^{\infty} \frac{P_m}{m} x^m$.

Now let $x_1, x_2, \ldots, x_n$~be independent variables over the
field of rational numbers $\mathbb{Q}$ and $K =\break =
\mathbb{Q}(x_1,x_2,\ldots,x_n)$, let $f(x) = \prod\limits_{i=1}^n (x
- x_i) = x^n + \sum\limits_{i=1}^n a_i x^{n-i}$. Since
$$
a_k =(-1)^k \sum\limits_{1 \leqslant i_1 < i_2 < \ldots < i_k \leqslant n} x_{i_1}x_{i_2}
\ldots x_{i_k} = (-1)^k s_k(x_1,x_2, \ldots, x_n), \; 1 \leqslant k \leqslant n,
$$
then by the lemma 5 it follows that the equality (11) is the polynomial
identity in the ring $\mathbb{Z}[x_1,x_2,\ldots, x_n]$. Lemma 5 also follows
from the equality (11) and from classical theorem about unique representation
of a symmetric polynomials from the ring $\mathbb{Z}[x_1,x_2,\ldots,x_n]$ in
the form of a polynomials $g(s_1,s_2,\ldots,s_n)$, where $g(x_1,x_2,\ldots,x_n)
\in \mathbb{Z}[x_1,x_2,\ldots,x_n]$. Now let $K$~be a field of positive
characteristic and $f(x) \in K[x],$ let $E$~be the splitting field of $f(x)$
over the field $K$, let $ f(x) = x^n + \sum\limits_{i=1}^n a_i x^{n-i} =
\prod\limits_{i=1}^n (x-y_i)$, $y_i \in E$, $1 \leqslant i \le n$. Let $e$~be
the unity of the field $K$, let $\varphi  \colon \mathbb{Z} \to E,$ $\varphi
(k) = ke$, $k \in \mathbb{Z}$. Let $\psi \colon \mathbb{Z}[x_1,x_2,\ldots,x_n]
\to E$,
$$
\psi\biggl(\sum
b_{\lam_1,\lam_2, \ldots, \lam_n} x_1^{\lam_1} x_2^{\lam_2} \ldots
x_k^{\lam_k} \biggr) = \sum \varphi  (b_{\lam_1,\lam_2, \ldots, \lam_n})
y_1^{\lam_1} y_2^{\lam_2} \ldots y_k^{\lam_k} ),
$$
$\sum b_{\lam_1,\lam_2, \ldots, \lam_n} x_1^{\lam_1} x_2^{\lam_2} \ldots
x_k^{\lam_k} \in \mathbb{Z}[x_1,x_2,\ldots,x_n]$. Since the mapping $\varphi  $
is the homomorphism of the ring $\mathbb{Z}$ into the ring $E$, then by the
lemma 6 it follows that $\psi$ is the homomorphism of the ring $
\mathbb{Z}[x_1,x_2,\ldots,x_n]$ into the ring $E$. Since $\psi(s_k(x_1, x_2,
\ldots, x_n)) = s_k(\psi(x_1), \psi(x_2), \ldots, \psi(x_n)) = \break =
s_k(y_1,y_2, \ldots, y_n) = (-1)^k a_k,$ $1 \leqslant k \leqslant n,$ then by
applying the homomorpism $\psi$ to the both sides of the polynomial equality in
the ring $\mathbb{Z}[x_1,x_2,\ldots,x_n]$
\begin{multline}
\sum_{i=1}^n x_i^m = (-1)^m \sum_{\substack{\lam_1 + 2\lam_2 + n
\lam_n = m, \\ \lam_i \in \mathbb{Z}, \, \lam_i \geqslant 0,\, 1
\leqslant i \leqslant n }} (-1)^{\lam_1+\lam_2+ \ldots + \lam_n }
\frac{m((\lam_1+\lam_2+
\ldots + \lam_n -1)! )}{\lam_1!\lam_2! \ldots \lam_n! } \times \\
\times \prod_{k=1}^n (s_k(x_1,\ldots,x_n))^{\lam_k}
\end{multline}
we obtain that
\begin{multline*}
\sum_{i=1}^n y_i^m = (-1)^m \sum_{\substack{\lam_1 + 2\lam_2 + n
\lam_n = m, \\ \lam_i \in \mathbb{Z}, \, \lam_i \geqslant 0,\, 1
\leqslant i \leqslant n }} (-1)^{\lam_1+\lam_2+ \ldots + \lam_n }
\frac{m((\lam_1+\lam_2+ \ldots + \lam_n -1)! )}{\lam_1!\lam_2!
\ldots \lam_n! } \times \\ \shoveright{\times \prod_{k=1}^n
(s_k(y_1,\ldots,y_n))^{\lam_k} =} \\
=(-1)^m \sum_{\substack{\lam_1+ 2\lam_2 + n \lam_n = m, \\ \lam_i
\in \mathbb{Z}, \, \lam_i \geqslant 0,\, 1 \leqslant i \leqslant n
}}  (-1)^{\lam_1+\lam_2+ \ldots + \lam_n } \frac{m((\lam_1+\lam_2+
\ldots + \lam_n -1)! )}{\lam_1!\lam_2! \ldots \lam_n!} \prod_{k=1}^n
((-1)^k a_k)^{\lam_k} =
 \\ = \sum_{\substack{\lam_1 + 2\lam_2 + n \lam_n = m, \\ \lam_i \in
\mathbb{Z}, \, \lam_i \geqslant 0,\, 1 \leqslant i \leqslant n }}
(-1)^{\lam_1+\lam_2+ \ldots + \lam_n } \frac{m((\lam_1+\lam_2+
\ldots + \lam_n -1)! )}{\lam_1!\lam_2! \ldots \lam_n! }
\prod_{k=1}^n (a_k)^{\lam_k}
\end{multline*}
Lemma 7 is proved.
\end{proof}

Since $s_k(x_1,x_2,\ldots,x_n,0, \ldots,0) = s_k(x_1,x_2,\ldots,x_n)$ if $k
\leqslant n$, and $s_k(x_1,x_2,\ldots \break \ldots, x_n,0\dots,0)=0$ if $k
> n$ then the equality (13) follows from [5,~(16),~p.~381]. Really,
from [5,~(16),~p.~381] it follows that
\begin{multline*}
\sum_{i=1}^n x_i^m + \sum_{j=1}^m y_j^m = \frac{1}{(-1)^{m-1}(m-1)!}
\sum_{\substack{\lam_1 + 2\lam_2 + m \lam_m = m,
\\ \lam_i \in \mathbb{Z}, \, \lam_i \geqslant 0,\, 1 \leqslant i
\leqslant m }} \frac{m!}{\prod\limits_{i=1}^m
(i!)^{\lam_i}(\lam_i!)} \times \\
\times (-1)^{\lam_1+\lam_2+ \ldots + \lam_m -1} (\lam_1+\lam_2+
\ldots + \lam_m -1)! \prod_{i=1}^m (i!
s_i(x_1,x_2,\ldots, x_n,y_1 \ldots, y_m))^{\lam_i} \!= \\
=(-1)^m \! \!\sum_{\hbox to 40pt{ $ \substack{\lam_1 + 2\lam_2 + m \lam_m = m, \\
\lam_i \in \mathbb{Z}, \, \lam_i \geqslant 0,\, 1 \leqslant i \leqslant m }$}}
(-1)^{\lam_1+\lam_2+ \ldots + \lam_m } \frac{m((\lam_1+\lam_2+ \ldots + \lam_m
-1)! )}{\lam_1!\lam_2! \ldots \lam_m! } \prod_{i=1}^m (s_i(x_1,\ldots,x_n,y_1,
\ldots, y_m ))^{\lam_i}.
\end{multline*}
Setting $y_1 = \ldots = y_m =0,$ we obtain (13).
\begin{lemma}
Let $K$~be a field and $f(x) \! \in \! K[x]$, let $E$~be the
splitting field of $f(x)$ over $K$, $f(x)= \break
=\prod\limits_{i=1}^n(x-x_i),$ $x_i \in E$ $x_i \ne 1$ for all $i, \
1\leqslant i \leqslant n.$ Then
\begin{equation}
\prod_{i=1}^n \ls x- \frac{1}{1-x_i} \rs = (f(1))^{-1} x^n f \ls 1-\frac{1}{x}\rs
\end{equation}
\end{lemma}

\begin{proof}
We have,
$$
\prod_{i=1}^n \ls x- \frac{1}{1-x_i} \rs = \ls \prod_{i=1}^n (1- x_i) \rs ^{-1} \cdot
\prod_{i=1}^n ((1- x_i)x-1) = (f(1))^{-1} g(x),
$$
where $g(x)=\prod\limits_{i=1}^n ((1-x_i)x-1).$ But
$$
x^n g \ls \frac{1}{x}\rs = \prod_{i=1}^n (1-x_i-x)=\prod_{i=1}^n ((1-x)-x_i)=f(1-x).
$$
Hence, $g(x)=x^n f \ls 1-\frac{1}{x} \rs$ and
$$
 \prod_{i=1}^n \ls x- \frac{1}{1-x_i} \rs
= (f(1))^{-1} x^n f \ls 1-\frac{1}{x}\rs
$$
\end{proof}
\begin{lemma}
Let $K$~be a field of characteristic $0$ or relatively prime with
$n$, if $char K \ne 0,$ $c \in K, \, c \ne 0,1;$ let $\vep$~be a
primitive root from $1$ of degree $n$, let $\alpha$~be an arbitrary
root of the polynomial $x^n-c,$
$$
f_{n,m}(k)=f_{n,m}(c, \alpha; k) = \sum_{i=0}^{n-1} \frac{\vep^{ik}}{(1-\vep^t
\alpha)^m}, \quad  n \geqslant 1, \, m \in \mathbb{Z}, \ k \in \mathbb{Z}
$$
Then
\begin{equation}
f_{n,m}(k)=\alpha^{-k} \lk (-1)^m {k-1 \choose m-1} n + \sum_{i=0}^{m-1}(-1)^i {k \choose
i} f_{n,m-i}(0) \rk, \, 1 \leqslant k
\leqslant n, \, m \geqslant 1
\end{equation}
\begin{multline}
f_{n,m}(0)=(-1)^m \sum_{\substack{\lam_1 + 2\lam_2 + \ldots + n\lam_n=m \\
\lam_i \in \mathbb{Z}, \, \lam_i \geqslant 0, 1 \leqslant i \leqslant n}}
\frac{m(\lam_1 + \lam_2 + \ldots + \lam_n-1)!}{\lam_1! \lam_2! \ldots \lam_n!}
\ls \frac{1}{c-1} \rs^{\lam_1 + \lam_2 + \ldots + \lam_n} \prod_{i=1}^n {n \choose i}^{\lam_i}, \\ m \geqslant 1
\end{multline}
\begin{gather}
f_{n,m}(0)=\frac{1}{c-1}\lk (-1)^m {n \choose m}m + \sum_{i=1}^{m-1} (-1)^{i} {n \choose i} f_{n,m-i}(0) \rk, \, 1 \leqslant m \leqslant n, \\
f_{n,m}(0)=\frac{1}{c-1}  \sum_{i=1}^{n} (-1)^{i} {n \choose i} f_{n,m-i}(0), \, m \geqslant n+1, \\
\sum_{m=1}^{\infty} (f_{n,m}(0)) x^m = \frac{nx(1-x)^{n-1}}{(1-x)^n-c}, \quad
\sum_{m=0}^{\infty} (f_{n,m}(0)) x^m = \frac{n[(1-x)^{n-1}-c]}{(1-x)^n-c}
\end{gather}
\end{lemma}

\begin{proof}
We shall prove the equality (15) by induction on $k, \, 1 \leqslant k \leqslant n.$ Let $k
\in \mathbb{Z},$  $m \in \mathbb{Z}.$ Then
\begin{multline*}
f_{n,m}(k)-\alpha^{-1}f_{n,m}(k-1)=\sum_{i=0}^{n-1}\frac{\vep^{ik}}{(1-\vep^i \alpha)^m}-
\alpha^{-1} \sum_{i=0}^{n-1}\frac{\vep^{i(k-1)}}{(1-\vep^i \alpha)^m}= \\
=\sum_{i=0}^{n-1}\frac{\vep^{i(k-1)}(\vep^i - \alpha^{-1})}{(1-\vep^i \alpha)^m}=
-\alpha^{-1}\sum_{i=0}^{n-1}\frac{\vep^{i(k-1)}(1-\vep^i \alpha)}{(1-\vep^i \alpha)^m}=
-\alpha^{-1}\sum_{i=0}^{n-1}\frac{\vep^{i(k-1)}}{(1-\vep^i \alpha)^{m-1}}= \\
= -\alpha^{-1}f_{n,m-1}(k-1).
\end{multline*}
Hence
\begin{equation}
f_{n,m}(k)= \alpha^{-1}[f_{n,m}(k-1)-f_{n,m-1}(k-1)], \, k \in \mathbb{Z}, \, m \in
\mathbb{Z}.
\end{equation}
Since $f_{n,0}(0)=n,$ then from equality (20) it follows that
the equality (15) is correct for $k=1.$

Case $m=1.$ Since
$$
\sum_{i=0}^{n-1} \vep^{ik} = \frac{1-\vep^{kn}}{1-\vep^k}=0,
$$
if $1 \leqslant k \leqslant n-1,$ then from the equality (20) it follows that
$$
f_{n,1}(k)=\alpha^{-1} \lk f_{n,1}(k-1)-\sum_{i=0}^{n-1} \vep^{i(k-1)} \rk =\begin{cases}
\alpha^{-1}f_{n,1}(k-1), &\text{ if $2 \leqslant k \leqslant n ,$}  \\
\alpha^{-1}(f_{n,1}(0)-n), &\text{ if $k=1.$ }
\end{cases}
$$
It follows that
$$
f_{n,1}(k)=\alpha^{-(k-1)}f_{n,1}(1), \mbox{ if $2 \leqslant k \leqslant n $.}
$$
Therefore $f_{n,1}(k)=\alpha^{-(k-1)}f_{n,1}(1),$ if $1 \leqslant k \leqslant n $. Since
$$
f_{n,1}(1)=\alpha^{-1}(f_{n,1}(0)-n),
$$
then $f_{n,1}(k)=\alpha^{-k}(f_{n,1}(0)-n),$ if $1 \leqslant k \leqslant n $.

Case $m \geqslant 2$. Let $ 2 \leqslant k \leqslant n$. From the equality (20)
and the induction hypothesis for $k-1$:
\begin{gather*}
f_{n,m}(k-1)=\alpha^{-(k-1)} \lk (-1)^m {k-2 \choose m-1}n +
\sum_{i=0}^{m-1} (-1)^i {k-1 \choose i} f_{n,m-i}(0) \rk , \\
f_{n,m-1}(k-1)=\alpha^{-(k-1)} \lk (-1)^{m-1} {k-2 \choose m-2}n + \sum_{i=0}^{m-2}
(-1)^i {k-1 \choose i} f_{n,m-1-i}(0) \rk
\end{gather*}
it follows that
\begin{multline*}
f_{n,m}(k) = \alpha^{-k} \biggl[ (-1)^m {k-2 \choose m-1}n - (-1)^{m-1}{k-2 \choose m-2}n + \\
+ \sum_{i=0}^{m-1} (-1)^i {k-1
\choose i}f_{n,m-i}(0) - \sum_{i=0}^{m-2} (-1)^i {k-1 \choose i} f_{n,m-1-i}(0) \biggr]  = \\
= \alpha^{-k} \biggl[ (-1)^m \ls {k-2 \choose m-1} + {k-2 \choose m-2}
\rs n + \sum_{i=0}^{m-1} (-1)^i {k-1 \choose i} f_{n,m-i}(0) - \\
- \sum_{j=1}^{m-1} (-1)^{j-1} {k-1 \choose j-1} f_{n,m-j}(0) \biggr]
=\alpha^{-k} \biggl[ (-1)^m {k-1 \choose m-1}n + f_{n,m}(0) + \\
+ \sum_{i=1}^{m-1} (-1)^i \ls {k-1 \choose i} + {k-1 \choose i-1} \rs f_{n,m-i}(0) \biggr] = \\
=\alpha^{-k} \lk (-1)^m {k-1 \choose m-1}n + f_{n,m}(0) +
\sum_{i=1}^{m-1} (-1)^i {k \choose i}f_{n,m-i}(0)\rk = \\
=\alpha^{-k} \lk (-1)^m {k-1 \choose m-1}n + \sum_{i=0}^{m-1} (-1)^i {k \choose
i}f_{n,m-i}(0)\rk .
\end{multline*}
From induction principle it follows that the equality (15) is
correct for all  $k$, $1 \leqslant k \leqslant n$, and $m \geqslant
1.$ Let
$$
\varphi(x)=x^n-c= \prod_{i=1}^n(x-x_i)=\prod_{i=1}^n(x- \vep^i \alpha).
$$
From the lemma 8 it follows that
\begin{multline*}
\prod_{j=0}^{n-1} \ls x-\frac{1}{1-\vep^j \alpha} \rs =\prod_{i=1}^{n}\ls
x-\frac{1}{1-\vep^i \alpha} \rs =
(\varphi (1))^{-1}x^n \varphi \ls 1-\frac{1}{x}\rs =\\
=(1-c)^{-1}x^n \lk \ls 1-\frac{1}{x} \rs^n-c \rk =(1-c)^{-1}[(x-1)^n-cx^n]=
x^n+(1-c)^{-1}\sum_{i=1}^n(-1)^i{n \choose i} x^{n-i}.
\end{multline*}
Hence and from lemma 7 it follows that
\begin{multline*}
f_{n,m}(0) = \sum_{i=0}^{n-1} \frac{1}{(1- \vep^i \alpha)^m}=
 \sum_{i=0}^{n-1} \ls \frac{1}{1- \vep^i \alpha} \rs^m =
\sum_{\substack{\lam_1 + 2\lam_2 + \ldots + n\lam_n = m \\
\lam_i \in \mathbb{Z}, \, \lam_i \geqslant 0, \, 1 \leqslant i \leqslant n}}
(-1)^{\lam_1 + \ldots + \lam_n} \times \\
\times \frac{m(\lam_1+ \lam_2 + \ldots + \lam_n-1)!}{\lam_1!
\lam_2! \ldots \lam_n!} \prod_{i=1}^n \lk (1-c)^{-1}(-1)^i {n \choose i} \rk^{\lam_i}=\\
=\sum_{\substack{\lam_1 + 2\lam_2 + \ldots + n\lam_n = m \\
\lam_i \in \mathbb{Z}, \, \lam_i \geqslant 0, \, 1 \leqslant i \leqslant n}} (-1)^{\lam_1 +
\lam_2 + \ldots + \lam_n} \frac{m(\lam_1 + \lam_2 + \ldots +
\lam_n-1)!}{\lam_1! \lam_2! \ldots \lam_n!} \times \\
\times \ls \frac{1}{1-c} \rs ^{\lam_1+ \lam_2 + \ldots + \lam_n}
(-1)^{\lam_1 + 2\lam_2 + \ldots + n\lam_n }\prod_{i=1}^n {n \choose i}^{\lam_i}=\\
=(-1)^m \sum_{\substack{\lam_1 + 2\lam_2 + \ldots + n\lam_n =m \\
\lam_i \in \mathbb{Z}, \, \lam_i \geqslant 0, \, 1 \leqslant i \leqslant n}} \ls
\frac{1}{c-1} \rs^{\lam_1+ \lam_2 + \ldots + \lam_n}
\frac{m(\lam_1 + \lam_2 + \ldots + \lam_n-1)!}{\lam_1! \lam_2!
\ldots \lam_n!} \prod_{i=1}^n {n \choose i}^{\lam_i}.
\end{multline*}
From the equality
\begin{equation}
\prod_{i=0}^{n-1}\ls x- \frac{1}{1- \vep^i \alpha} \rs = x^n + (1-c)^{-1}\sum_{i=1}^n
(-1)^i {n \choose i} x^{n-i}
\end{equation}
and respectively from first and second formulas of Newton it follows the
equalities (17) and (18). How it is no wonder that, the equalities (17) and (18)
also follows directly
 from the equality (15). Really, setting in equality (15) $k=n$, we obtain
\begin{equation}
f_{n,m}(n)=\alpha^{-n} \lk (-1)^m {n-1 \choose m-1} n + f_{n,m} (0) +
\sum_{i=1}^{m-1}(-1)^i {n \choose i} f_{n,m-i}(0) \rk ,
\end{equation}
Since $f_{n,m}(n) = f_{n,m}(0), \alpha^{-n} = c^{-1}$ and ${n-1 \choose m-1}n={n-1
\choose m-1} \frac{n}{m} m = {n \choose m} m,$ then from the equality (22) it follows
 that
$$
cf_{n,m}(0)=(-1)^m {n \choose m} m + f_{n,m}(0) + \sum_{i=1}^{m-1}(-1)^i {n \choose i}
f_{n,m-i}(0),
$$
and therefore
\begin{equation}
f_{n,m}(0)=\frac{1}{c-1} \lk (-1)^m {n \choose m} \cdot m  + \sum_{i=1}^{m-1}(-1)^i {n
\choose i} f_{n,m-i}(0) \rk,
\end{equation}
for all $m \geqslant 1.$

The equalities (17) and (18) are equivalent one equality (23) for all $m \geqslant 1.$
Let
$$
f(x)=\prod_{i=0}^{n-1}\biggl(x-\frac{1}{1-\vep^i \alpha}\biggr).
$$
Then, how we show above, $ f(x)=(1-c)^{-1}[(x-1)^n-cx^n]$ and it follows that
$$
g(x)=x^n f(x^{-1})=(1-c)^{-1}[(1-x)^n-c].
$$
Hence and by the lemma 1 it follows that
$$
\sum_{m=1}^{\infty} f_{n,m} (0) x^m = -x \ls \frac{\dif}{\dif x} g(x) \rs (g(x))^{-1}=
\frac{nx(1-x)^{n-1}}{(1-x)^n - c}.
$$
Since $f_{n,0}(0)=n,$ then
$$
\sum_{m=0}^{\infty} f_{n,m}(0)x^m = n + \sum_{m=1}^{\infty}
f_{n,m}(0)x^m = n + \frac{nx(1-x)^{n-1}}{(1-x)^n - c}=
\frac{n[(1-x)^{n-1}-c]}{(1-x)^n - c}.
$$
\end{proof}
\begin{theorem}
Let $K$~be a field of characteristic $0$ or relatively prime with
$n,$ if $\charr K \ne 0 $. Let $a,b \in K \setminus \{ 0 \}, \,  a
\ne b,$ $\vep$~be a primitive root from~$1$ of degree $n.$ Let
$\beta$~and~$\gamma$~be an arbitrary roots of the polynomials $x^n +
a$ and $x^n + b$ in the splitting field of $(x^n +a)(x^n +b)$ over
$K$, respectively. Let $ x_i= \vep^i \beta$ and $ y_i= \vep^i
\gamma,$ $1 \leqslant i \leqslant n,$ let $\alpha =
\beta^{-1}\gamma, $ let $ f_{n, m}(k)=
\sum\limits_{i=0}^{n-1}\frac{\vep^{ik}}{(1-\vep^i \alpha)^m},$ $ k
\in \mathbb{Z},$ $m \geqslant 1. $ Then
\begin{multline}
\det\ls \ls \frac{1}{(x_i-y_j)^m}\rs_{1 \leqslant i,j \leqslant n} \rs=
(-1)^{mn} a ^{-m} \alpha^{- \frac{n(n-1)}{2}} \times \\
\times f_{n,m}(0) \prod_{k=1}^{n-1} \lk (-1)^m {k-1 \choose m-1}n + \sum_{i=0}^{m-1}
(-1)^i{k \choose i} f_{n, m-i}(0) \rk
\end{multline}
\end{theorem}

\begin{proof}
Carry out from $i$-th row, $1 \leqslant i \leqslant n,$ of matrix $\bigl(\frac{1}{(x_i -
y_j)^m}\bigr)_{1\leqslant i,j \leqslant n}$ factor $x_i^{-m}$ we obtain that
\begin{multline}
\det\ls \ls \frac{1}{(x_i - y_j)^m}\rs_{1 \leqslant i, j \leqslant n} \rs = (x_1 x_2
\ldots x_n)^{-m}\det\ls \ls \frac{1}{(1 -
x_i^{-1}x_j)^m} \rs _{1 \leqslant i, j \leqslant n}\rs =\\
= (x_1 x_2 \ldots x_n)^{-m}\det\ls \ls \frac{1}{(1 - \vep^{j-i} \alpha)^m} \rs_{1
\leqslant i, j \leqslant n}\rs.
\end{multline}

Let $b_k=\frac{1}{(1-\vep^k \alpha)^m}, \, 1\leqslant k\leqslant n.$ Then from (25)
it follows that
\begin{equation}
\det\ls \ls \frac{1}{(x_i - y_j)^m} \rs_{1 \leqslant i,j \leqslant n}\rs = (x_1 x_2
\ldots x_n)^{-m} \det ((b_{(j-i) \mod n})_{1 \leqslant i,j \leqslant n}).
\end{equation}
From well known formula for the determinant group matrix of
finite cyclic group it follows that
\begin{equation}
\det ((b_{(j-i)mod n})_{1 \leqslant i,j \leqslant n})=\prod_{k=1}^n \sum_{i=1}^n
\vep^{ik} b_i.
\end{equation}
Since
$$
x^n+a=\prod_{i=1}^n (x- \vep^{i} \beta) = \prod_{i=1}^{n}(x-x_i),
$$
then $x_1 x_2 \ldots x_n=(-1)^n a$ and from the equalities (26), (27)
and (15) it follows that
\begin{multline*}
\det\ls \ls \frac{1}{(x_i - y_j)^m} \rs_{1 \leqslant i, j \leqslant n} \rs = ((-1)^n
a)^{-m} \prod_{k=1}^n
\sum_{i=1}^n \vep^{ik} b_i=\\
=(-1)^{mn}a^{-m}\prod_{k=1}^n \sum_{i=1}^n \frac{\vep^{ik}}{(1-\vep^i \alpha)^m}
=(-1)^{mn}a^{-m}\prod_{k=1}^n f_{n,m}(k)= (-1)^{mn}a^{-m}\prod_{k=0}^{n-1} f_{n,m}(k)= \\
=(-1)^{mn}a^{-m} f_{n,m}(0) \prod_{k=1}^{n-1} f_{n,m}(k)= \\
=(-1)^{mn} a^{-m}f_{n,m}(0) \prod_{k=1}^{n-1} \alpha^{-k}
\lk (-1)^m {k-1 \choose m-1}n + \sum_{i=0}^{m-1} (-1)^i{k \choose i} f_{n, m-i}(0) \rk= \\
=(-1)^{mn} a^{-m}f_{n,m}(0) \alpha^{-\sum_{k=1}^{n-1}k} \prod_{k=1}^{n-1}
\lk (-1)^m {k-1 \choose m-1}n + \sum_{i=0}^{m-1} (-1)^i{k \choose i} f_{n, m-i}(0) \rk= \\
=(-1)^{mn} a^{-m} \alpha^{-\frac{n(n-1)}{2}} f_{n,m}(0) \prod_{k=1}^{n-1} \lk (-1)^m {k-1
\choose m-1}n + \sum_{i=0}^{m-1} (-1)^i{k \choose i} f_{n, m-i}(0) \rk.
\end{multline*}
Theorem 2 is proved.
\end{proof}

\begin{theorem}
Let $K$~be a field of characteristic $0$ or relatively prime with
$n,$ if $\charr K \ne 0.$ Let $a,b \in K \setminus \{0\}, \, a \ne
b,$ $x_1, \ldots, x_n$ and $y_1, \ldots, y_n$~are the distinct roots
of the polynomials $x^n+a$ and $x^n+b$ in the splitting field of
$(x^n + a)(x^n + b)$ over $K$, respectively. Then
\begin{gather}
\per \ls \ls \frac{1}{x_i-y_j}\rs_{1 \leqslant i,j \leqslant n}\rs =\frac{n}{(b-a)^n}
\prod_{k=1}^{n-1}[na+k(b-a)] =\\
= \frac{n}{(b-a)^n} \prod_{k=1}^{n-1}[nb+k(a-b)] =
\end{gather}
\begin{gather}
 =\begin{cases}
(-1)^{\frac{n-1}{2}} \cfrac{n}{(b-a)^n}
\prod\limits_{k=1}^{\frac{n-1}{2}}[-na-k(b-a)][nb-k(b-a)],
\mbox{if $n \equiv 1 ( \mod 2)$,} \\
\cfrac{n}{2} \cdot \cfrac{n(a+b)}{(b-a)^n}
\prod\limits_{k=1}^{\frac{n}{2}-1}[na+k(b-a)][nb+k(a-b)], \mbox{if $n\equiv 0(\mod 2)$}.
\end{cases}
\end{gather}
\end{theorem}

\begin{proof}
From the Borchardt's theorem [4, p.18] it follows that
\begin{equation}
\per\ls \ls\frac{1}{x_i-y_j}\rs_{1 \leqslant i,j \leqslant n}\rs \cdot \det\ls
\ls\frac{1}{x_i-y_j}\rs_{1 \leqslant i,j \leqslant n}\rs =\det\ls
\ls\frac{1}{(x_i-y_j)^2}\rs_{1 \leqslant i,j \leqslant n}\rs.
\end{equation}
Since $\per\ls(\frac{1}{x_i-y_j})_{1 \leqslant i,j \leqslant n}\rs$ is independent
from the order of the roots of the polynomials $x^n+a$~and $y^n+b,$ then choose ordering the
same as in theorem 2 and using introduced in the same place notations we obtain from (24)
that
\begin{gather}
\det\ls \ls \frac{1}{x_i-y_j}\rs_{1 \leqslant i,j \leqslant n}\rs =(-1)^n a^{-1}
\alpha^{-\frac{n(n-1)}{2}}f_{n,1}(0)\prod_{k=1}^{n-1}[f_{n,1}(0)-n],\\
\det\ls \ls\frac{1}{(x_i-y_j)^2}\rs_{1 \leqslant i,j \leqslant n} \rs = a^{-2}
\alpha^{\frac{n(n-1)}{2}}f_{n,2}(0)\prod_{k=1}^{n-1}[f_{n,2}(0)-kf_{n,1}(0)+(k-1)n],
\end{gather}
Since
$$
\alpha^n=(\beta^{-1}\gamma)^n=(\beta^n)^{-1} \cdot \gamma^n=(-a)^{-1} \cdot (-b) = a^{-1}
b=c,
$$
then from the equality (16) it follows that
\begin{gather}
f_{n,1}(0)=-\frac{n}{c-1}=-\frac{n}{\frac{b}{a}-1}=-\frac{na}{b-a}, \\
f_{n,2}(0)=\ls \frac{1}{c-1} \rs^2 \cdot {n \choose 1}^2 + 2 \frac{1}{c-1}{n \choose 2} =
\frac{n^2}{(\frac{b}{a}-1)^2}+\frac{n(n-1)}{\frac{b}{a}-1}=\frac{n^2
a^2}{(b-a)^2}+\frac{n(n-1)a}{b-a}
\end{gather}
From the equalities (31) --- (35) it follows that $\per((\frac{1}{x_i-y_j})_{1 \leqslant i,j
\leqslant n})=$
\begin{multline*}
=(-1)^n a^{-1} \frac{\frac{n^2 a^2}{(b-a)^2}+\frac{n(n-1)a}{b-a}}{-\frac{na}{b-a}} \cdot
\prod_{k=1}^{n-1} \frac{\frac{n^2
a^2}{(b-a)^2}+\frac{n(n-1)a}{b-a}+\frac{kna}{b-a}+(k-1)n}
{-\frac{na}{b-a}-n}=\\
=(-1)^{n-1} a^{-1}\lk \frac{na}{b-a}+n-1 \rk \prod_{k=1}^{n-1}
\frac{\frac{na}{b-a} [\frac{na}{b-a}+n-1+k+\frac{(k-1)(b-a)}{a}]}
{-\frac{nb}{b-a}}= \\
 =a^{-1}\ls \frac{a}{b} \rs^{n-1} \cdot
\frac{(n-1)b+a}{b-a}\prod_{k=1}^{n-1} \lk
\ls\frac{na}{b-a}+n-1\rs+k+\frac{(k-1)(b-a)}{a} \rk =
\end{multline*}
\begin{multline*}
 =a^{-1}\ls
\frac{a}{b} \rs^{n-1} \cdot \frac{(n-1)b+a}{b-a}\prod_{k=1}^{n-1}
\lk \frac{nb-(b-a)}{b-a} + \frac{(k-1)b+a}{a} \rk =\\
=a^{-1}\ls \frac{a}{b} \rs^{n-1} \cdot \frac{(n-1)b+a}{b-a}\prod_{k=1}^{n-1}
\lk \frac{nb}{b-a} + \frac{(k-1)b}{a}\rk=\\
=a^{-1}\ls\frac{a}{b}\rs^{n-1} \cdot \frac{(n-1)b+a}{b-a}\prod_{k=1}^{n-1}
\frac{b}{a(b-a)}[na + (k-1)(b-a)]=\\
=a^{-1} \frac{(n-1)b+a}{(b-a)^n}\prod_{k=0}^{n-2} [na+k(b-a)]=
\frac{n}{(b-a)^n}\prod_{k=1}^{n-1}[na+k(b-a)].
\end{multline*}

Since the mapping $\varphi(k)=n-k$ is the bijection of the set
$\{1,2,\ldots,n-1\}$ onto itself, then
$$
\prod_{k=1}^{n-1}[na+k(b-a)]=\prod_{k=1}^{n-1}[na+\varphi (k)(b-a)]=
\prod_{k=1}^{n-1}[na+(n-k)(b-a)]=\prod_{k=1}^{n-1}[nb-k(b-a)].
$$
Let $n \equiv 1 (\mod 2), \, n\geqslant 3.$ Then the mapping
$\varphi(k)=n-k$ is a bijection of the set
$\{1,2,\ldots,\frac{n-1}{2}\}$ onto the set
$\{\frac{n+1}{2},\frac{n+1}{2}+1,\ldots,n-1\}$ and therefore
\begin{multline}
\prod_{k=1}^{n-1}[na+k(b-a)]=\ls \prod_{k=1}^{(n-1)/2}[na+k(b-a)]\rs \cdot
\prod_{k=\frac{n+1}{2}}^{n-1}[na+k(b-a)]= \\
=\ls \prod_{k=1}^{(n-1)/2}[na+k(b-a)] \rs \cdot
\prod_{k=1}^{(n-1)/2}[na+ \varphi (k)(b-a)]= \\
= \ls \prod_{k=1}^{(n-1)/2}[na+k(b-a)]\rs \prod_{k=1}^{(n-1)/2}[nb- k(b-a)]=\\
(-1)^{\frac{n-1}{2}} \prod_{k=1}^{\frac{n-1}{2}}[-na-k(b-a)][nb- k(b-a)].
\end{multline}
Let $n\equiv 0 (\mod 2).$ Then $\varphi(k)=n-k$ is the bijection of
the set $\{1,2,\ldots,\frac{n}{2}-1\}$ on the set
$\{\frac{n}{2}+1,\frac{n}{2}+2,\ldots,n-1\}$ and therefore
\begin{multline*}
\prod_{k=1}^{n-1}[na+k(b-a)]=\ls \prod_{k=1}^{(n/2)-1}[na+k(b-a)]\rs \cdot [na +
\frac{n}{2}(b-a)] \cdot \prod_{k=\frac{n}{2}+1}^{n-1}[na+k(b-a)]=\\
=\prod_{k=1}^{n/2-1}[na+k(b-a)] \cdot \frac{n}{2} (a+b) \cdot \prod_{k=1}^{n/2-1}[na+
\varphi (k)(b-a)]= \frac{n}{2}(a+b) \prod_{k=1}^{n/2-1}[na+k(b-a)][nb-k(b-a)].
\end{multline*}
Theorem 3 is proved.
\end{proof}
Theorem 3 is a direct generalization of a conjecture of R.F.Scott(1881) if $a=-1, \, b=1.$
\begin{theorem}
Let $K$~be a field of characteristic $0$ or relatively prime with $2n$ if
$\charr K \ne 0.$ Let $a\in K \setminus \{0\}$ and $x_1, \ldots, x_n$ and $y_1,
\ldots, y_n$~are distinct roots of the polynomials $x^n+a$ and $y^n-a$in the
splitting field of $(x^n+a)(x^n-a)$ over $K$, respectively. Then
$$
\per\ls \ls \frac{1}{x_i-y_j} \rs_{1 \leqslant i,j \leqslant n}\rs =\begin{cases}
(-1)^{\frac{n+1}{2}}\frac{n}{2^n a} \ls \prod_{k=1}^{(n-1)/2}(n-2k)\rs ^2, &\text{if
$ n \equiv 1(\mod 2)$} \\
0, &\text{if $n \equiv 0 (\mod 2)$}.
\end{cases}
$$
\end{theorem}

\begin{proof}
Setting in equality (30) $b=-a$ we obtain that
\begin{gather*}
\per \ls \ls \frac{1}{x_i-y_j}\rs _{1 \leqslant i,j \leqslant n}\rs =\begin{cases}
(-1)^{\frac{n-1}{2}}\frac{n(-a)^{n-1}}{(-2a)^n}\prod_{k=1}^{(n-1)/2}(n-2k)^2, &\text{if
$ n \equiv 1(\mod 2)$} \\
0, &\text{if $n \equiv 0 (\mod 2)$}.
\end{cases} \\
=\begin{cases}
(-1)^{\frac{n+1}{2}}\frac{n}{2^n a} \ls \prod_{k=1}^{(n-1)/2}(n-2k)\rs^2, &\text{if $ n \equiv 1(\mod 2)$} \\
0, &\text{if $n \equiv 0 (\mod 2)$}.
\end{cases}
\end{gather*}
\end{proof}


\begin{theorem}
Let $f(x)$ and $g(x)$~be separable polynomials over a field $K$ and
$x_1, \ldots, x_n$ and $y_1, \ldots, y_n$~are the distinct roots of
the polynomials $f(x)$ and $g(x)$, respectively, in the splitting
field $E$ of the polynomial $f(x)g(x)$ over the field $K$. Assume
that $m \leqslant n$ and the polynomials $f(x)$ and $g(x)$ are
relatively prime.
Then $\per\ls(\frac{1}{(x_i-y_j)^k})_{\substack{1 \leqslant i \leqslant m \\
1\leqslant j \leqslant n}} \rs \in K $ for all $k\in \mathbb{Z}$.
\end{theorem}

\begin{proof}
Let $a_{i,j}=\frac{1}{(x_i-y_j)^k}, \, 1 \leqslant i \leqslant m, \,
1 \leqslant j \leqslant n, \, A=(a_{i,j})_{\substack{1 \leqslant i
\leqslant m \\ 1 \leqslant j \leqslant n}}.$ Let $\Gal(E/K)$~be the
Galois group of all $K$~---~automorphisms of the field $E$ and
$\sigma \in \Gal(E/K).$ Since
$$
\{\sigma(x_1), \ldots, \sigma(x_m) \}=\{ x_1, \ldots, x_m \}, \,  \{\sigma(y_1), \ldots,
\sigma(y_n) \}=\{ y_1, \ldots, y_n \},
$$
then there exists substitutions $\varphi \in \SSym(m),$ $\psi \in
\SSym(n)$ such that
$$
\sigma(x_i)=x_{\varphi(i)}, \, \sigma(y_j)=y_{\psi(j)}, \, 1\leqslant i \leqslant m,\, 1
\leqslant j\leqslant n ,
$$
where $\SSym(n)$~is symmetric group of substitutions of degree~$n$. Therefore
\begin{multline*}
\sigma(\per(A))=\sigma \ls \per \ls \ls \frac{1}{(x_i - y_j)^k}\rs_{\substack{1 \leqslant i \leqslant m \\
1 \leqslant j \leqslant n}} \rs \rs = \per \ls \ls \frac{1}{(\sigma(x_i) -
\sigma(y_j))^k}\rs_{\substack{1 \leqslant i \leqslant m \\ 1 \leqslant
 j \leqslant n}}\rs =\\
  \per\ls \ls \frac{1}{(x_{\varphi(i)} - y_{\psi(j)})^k} \rs_{\substack{1 \leqslant i \leqslant m \\ 1 \leqslant
 j \leqslant n}}\rs = \per((a_{\varphi(i),\psi(j)})_{\substack{1 \leqslant i \leqslant m \\ 1 \leqslant
 j \leqslant n}})=\per(A)
\end{multline*}
So, for all $\sigma \in \Gal(E/K)$
\begin{equation}
\sigma(\per(A))=\per(A).
\end{equation}
Since the elements of the field $E$ which are invariant relatively all
$K$---automorphisms of the group
$\Gal(E/K)$ of the field $E$ belong to the field $K,$ then from (37) it
follows that $\per(A)
\in K.$

Theorem 5 is proved.
\end{proof}

\begin{theorem}
Let $f(x)$ and $g(x)$~be separable polynomials of degree $n$ over
the field $K$ and $x_1, \ldots, x_n$ and $y_1, \ldots, y_n$~are the
distinct roots of the polynomials $f(x)$ and $g(x)$, respectively,
in the splitting field $E$ of the polynomial $f(x)g(x)$ over the
field $K$. Assume that the polynomials $f(x)$ and $g(x)$ are
relatively prime. Then $ \bigl( \det \bigl( \bigl(
\frac{1}{(x_i-y_j)^k} \bigr)_{1 \leqslant i, j \leqslant n} \bigr)
\bigl)^2 \in K $ for all $k\in \mathbb{Z}$.
\end{theorem}

\begin{proof}
In the notations of the proof of the theorem 5 we have
\begin{multline*}
\sigma(\det(A))=\sigma \ls \det \ls \ls \frac{1}{(x_i - y_j)^k} \rs_{1 \leqslant i, j
\leqslant n} \rs \rs = \det \ls \ls \frac{1}{(\sigma(x_i) -
\sigma(y_j))^k}\rs_{1 \leqslant i, j \leqslant n} \rs= \\
 =\det \ls \ls \frac{1}{(x_{\varphi(i)} - y_{\psi(j)})^k}\rs_{1 \leqslant i, j \leqslant n}\rs
 = \det((a_{\varphi(i),\psi(j)})_{1 \leqslant i, j \leqslant n})= \\
= (\sgn \varphi)(\sgn \psi)\det((a_{i,j})_{1 \leqslant i, j \leqslant n})
 =(\sgn \varphi)(\sgn \psi)\det(A)
\end{multline*}
So, for all $\sigma \in \Gal(E/K)$
\begin{equation}
\sigma(\det(A))=\pm \det(A).
\end{equation}
From (38) it follows that $ \sigma((\det(A))^2)=(\det(A))^2$ for all
$\sigma \in \Gal(E/K)$ and therefore $(\det(A))^2\in K$

Theorem 6 is proved.
\end{proof}

\begin{lemma}
Let $m \geqslant 1, \, n \geqslant m+1, \ x_1, \ldots, x_n,\ y_1,
\ldots, y_n$ be elements in a commutative ring,
$a_{i,j}=(x_i-y_j)^m, \, A_n^{(m)} = (a_{i,j})_{1 \leqslant i,j
\leqslant n}.$ Then
\begin{equation}
 \det(A_n^{(m)}) = \begin{cases} \ls \prod_{i=0}^{n-1}{n-1 \choose i}\rs \prod_{1\leqslant
i<j\leqslant n}(x_j-x_i)(y_j-y_i),
&\text{if $n=m+1,$}\\
0, &\text{if $n \geqslant m+2$}
\end{cases}
\end{equation}
\end{lemma}

\begin{proof}
Let $b_{i,j} = {m \choose j-1}x_i^{j-1}, c_{i,j}=(-y_j)^{m-i+1}, \, 1 \leqslant i, j
\leqslant n,$
$$
B_{n,m}=(b_{i,j})_{1\leqslant i, j \leqslant n}, \, C_{n,m}=(c_{i,j})_{1\leqslant i, j
\leqslant n}.
$$
Since $m\leqslant n-1,$ that
\begin{multline*}
\sum_{k=1}^n b_{i,k} c_{k,j}=\sum_{k=1}^n  {m \choose k-1}x_i^{k-1} (-y_j)^{m-k+1}=
\sum_{l=0}^{n-1}  {m \choose l} x_i^{l} (-y_j)^{m-l}= \\
=\sum_{l=0}^{m}  {m \choose l} x_i^{l} (-y_j)^{m-l}=(x_i-y_j)^m=a_{i,j}.
\end{multline*}
Hence, $A_n^{(m)}=B_{n,m} \cdot C_{n,m}$ and therefore
\begin{equation}
\det(A_n^{(m)}) = \det(B_{n,m}) \cdot \det(C_{n,m})
\end{equation}
Since
\begin{equation*}
\det(B_{n,m}) = \ls \prod_{j=1}^n {m \choose j-1}\rs \det((x_i^{j-1})_{1 \leqslant i,j
\leqslant n})=
\ls \prod_{i=0}^{n-1}{m \choose i}\rs \cdot \prod_{1 \leqslant i < j \leqslant n}(x_j-x_i),
\end{equation*}
\begin{multline*}
\det(C_{n,m})=(-y_1)^m (-y_2)^m \ldots (-y_n)^m \cdot \prod_{1
\leqslant i < j \leqslant n}((-y_j^{-1})
- (-y_i^{-1})) = \\
=\ls \prod_{i=1}^n(-y_i)^m \rs \cdot \ls \prod_{1 \leqslant i < j \leqslant n} (y_j -
y_i) \rs \cdot \prod_{1 \leqslant
i < j \leqslant n}(y_i y_j)^{-1} = \\
=\ls \prod_{i=1}^n (-y_i)^m \rs \cdot \ls \prod_{1\leqslant i < j \leqslant n} (y_j - y_i)\rs \cdot \ls \prod_{i=1}^{n} y_i^{n-1} \rs ^{-1}=\\
=(-1)^{mn}\ls \prod_{i=1}^n y_i^{m-n+1}\rs \cdot \prod_{1\leqslant i < j \leqslant n}
(y_j - y_i),
\end{multline*}
then from (40) it follows that
\begin{equation}
\det(A_n^{(m)})=(-1)^{mn}\ls \prod_{i=0}^{n-1} {m \choose i} \rs \cdot \ls \prod_{i=1}^n
y_i^{m-n+1} \rs \cdot \prod_{1\leqslant i < j \leqslant n} (x_j-x_i)(y_j - y_i).
\end{equation}
The equality (39) follows directly from (41).
\end{proof}

\begin{lemma}
Let $n \geqslant m+1, \, m \geqslant 1.$ Then
\begin{equation}
\rank(((x_i - y_j)^m)_{1 \leqslant i,j \leqslant n}) = m+1 \quad
\mbox{ if $x_i \ne x_j, \, y_i \ne y_j$}
\end{equation}
for all $i,j, \, 1 \leqslant i < j \leqslant n.$
\end{lemma}
Proof follows directly from the equality (39).

From the lemma 11 it follows that the theorem of Carlitz-Levine [4, p.19]
is nonapplicable to a matrix
$((x_i - y_j)^{-m})_{1 \leqslant i, j \leqslant n}$ if $m \geqslant 2$,
 $n \geqslant m+1$ and $x_i \ne y_j$, $1 \leqslant i,j \leqslant n$, $x_i \ne x_j$,
$y_i \ne y_j$, $1 \leqslant i< j \leqslant n$.

\begin{lemma}
Let $K$~be a field of characteristic 0, $\varphi (x) \in K[[x]], \,
\varphi(0) =1, \, \exp(x)=\sum\limits_{n=0}^{\infty}\frac{x^n}{n!}.$
Then
\begin{equation}
\exp (\log (\varphi (x)))= \varphi (x).
\end{equation}
\end{lemma}
\begin{proof}
Let $F(x)=\exp (\log (\varphi (x)))$. Since $\frac{\dif}{\dif x}
\exp (x)= \exp (x),$ then
\begin{multline*}
\frac{\dif} {\dif x} F(x) = \frac{\dif}{\dif x} \exp (\log (\varphi
(x)))= \exp (\log (\varphi (x))) \frac{\dif}{\dif x} \log (\varphi
(x)) = F(x) (\varphi (x))^{-1}\frac{\dif}{\dif x} \varphi (x).
\end{multline*}
 $F(x) \frac{\dif}{\dif x} \varphi (x) - (\frac{\dif}{\dif x}F(x)) \varphi (x) = 0$.
Hence
$$
\frac{\dif}{\dif x}(F(x) (\varphi (x))^{-1})= \lk
\ls\frac{\dif}{\dif x} F(x)\rs \varphi (x) - F(x) \frac{\dif}{\dif
x} \varphi (x) \rk (\varphi (x))^{-2}=0
$$
 and therefore  $F(x)(\varphi (x))^{-1}=F(0)(\varphi (0))^{-1}=1, \, F(x)=\varphi(x).$
\end{proof}
\begin{lemma}
Let $K$~be a field of a characteristic $0$, $E$~be th splitting
field of $f(x)$ over $K$, $f(x) \in K[x]$,
$$
f(x) = x^n + \sum_{k=1}^n a_kx^{n-k} = \prod_{i=1}^n(x-x_i), \; x_i
\in E, \; 1 \leqslant i \leqslant n.
$$
Let $a_k=0$ for all $ k \geqslant n+1.$ Then for all $k \geqslant 1$
\begin{equation}
a_k = \sum_{\lam_1+2\lam_2+ \ldots + k \lam_k = k}\frac{(-1)^{\lam_1
+ \lam_2 + \ldots + \lam_k}} {1^{\lam_1} 2^{\lam_2} \ldots
k^{\lam_k} \lam_1! \lam_2! \ldots \lam_k!} \prod_{i=1}^{k}(x_1^i+
x_2^i + \ldots + x_n^i )^{\lam_i}.
\end{equation}
\end{lemma}
\begin{proof}
Let
$$
\varphi (x) = x^n f(x^{-1})= 1+ \sum_{k=1}^n a_k x^k = 1+ \sum_{k=1}^{\infty} a_k x^k.
$$
Then $\varphi (x) = \prod\limits_{i=1}^n (1-x x_i)$ and from
lemma 3 it follows that
\begin{multline*}
\exp (\log( \varphi(x))) = \sum_{s=0}^{\infty} \frac{1}{s!} \ls \log
\ls \prod_{j=1}^n
(1-x_j x)\rs \rs^s = \\
=\sum_{s=0}^{\infty} \frac{1}{s!} \ls \sum_{j=1}^n \log (1-x_j x)\rs
^s = \sum_{s=0}^{\infty} \frac{1}{s!} \ls -\sum_{j=1}^n
\sum_{i=1}^{\infty} \ls \frac{(x_j x)^i}{i} \rs \rs ^s.
\end{multline*}
Hence and from the lemma 12 it follows that
\begin{equation}
1+\sum_{k=1}^{\infty} a_k x^k = \sum_{s=0}^{\infty} \frac{1}{s!} \ls
-\sum_{j=1}^n \sum_{i=1}^{\infty} \frac{x_j^i}{i} x^i \rs ^s =
\sum_{s=0}^{\infty} \frac{1}{s!} \ls -\sum_{i=1}^{\infty}
\frac{1}{i} \ls \sum_{j=1}^n x_j^i \rs x^i \rs ^s.
\end{equation}
From the equality (45) it follows that
\begin{multline*}
a_k=\Coef\limits_{x^k} \sum_{s=1}^k \frac{1}{s!} \ls -\sum_{i=1}^k \frac{1}{i} \ls \sum_{j=1}^n x_j^i \rs x^i \rs^s = \\
=\sum_{s=1}^k \frac{(-1)^s}{s!} \sum_{\substack{\lam_1+\lam_2+\ldots+\lam_k=s \\
\lam_1+2\lam_2+\ldots+k\lam_k=k \\ \lambda_i \in \mathbb{Z}, \, \lambda_i \geqslant 0, \,
1 \leqslant i \leqslant k }}
\frac{s!}{\lam_1! \lam_2! \ldots \lam_k!}  \ls \prod_{i=1}^k \ls \frac{1}{i} \sum_{j=1}^n x_j^i \rs ^{\lam_i} \rs =\\
=\sum_{\substack{\lam_1+2\lam_2+\ldots+k\lam_k=k \\ \lambda_i \in \mathbb{Z}, \, \lambda_i \geqslant 0, \,
1 \leqslant i \leqslant k }}
\frac{(-1)^{\lam_1+\lam_2+\ldots+\lam_k}}{\lam_1!
\lam_2! \ldots \lam_k!} \prod_{i=1}^k \frac{1}{i^{\lam_i}} \ls \sum_{j=1}^n x_j^i \rs^{\lam_i}= \\
= \sum_{\substack{\lam_1+2\lam_2+ \ldots + k \lam_k = k \\ \lambda_i \in \mathbb{Z}, \, \lambda_i \geqslant 0, \,
1 \leqslant i \leqslant k }}\frac{(-1)^{\lam_1 +
\lam_2 + \ldots + \lam_k}} {1^{\lam_1} 2^{\lam_2} \ldots k^{\lam_k}
\lam_1! \lam_2! \ldots \lam_k!} \prod_{i=1}^{k}(x_1^i+ x_2^i +
\ldots + x_n^i )^{\lam_i}.
\end{multline*}

\end{proof}

\begin{lemma}
Let $A$~be a square matrix of order $n$ over a field of
characteristic $0$,
$$
\det(xI_n-A)=x^n+\sum_{k=1}^n a_k x^{n-k}.
$$
Then for all $k \geqslant 1$
\begin{equation}
a_k=\sum_{ \substack{\lam_1+2\lam_2+ \ldots + k \lam_k = k \\ \lambda_i \in \mathbb{Z}, \, \lambda_i \geqslant 0, \,
1 \leqslant i \leqslant k }}\frac{(-1)^{\lam_1 +
\lam_2 + \ldots + \lam_k } }{1^{\lam_1} 2^{\lam_2} \ldots k^{\lam_k}
\lam_1! \lam_2! \ldots \lam_k!} \prod_{i=1}^{k} (\Tr (A^i))^{\lam_i}
\end{equation}
\end{lemma}
\begin{proof}
Let $x_1, x_2, \ldots, x_n$~be the eigenvalues of the matrix $A$.
Then
$$
\det (xI_n-A)=\prod_{i=1}^n(x-x_i), \quad \Tr(A^i)=\sum_{j=1}^nx_j^i
$$
for all $i \geqslant 0$
 and the equality (46) follows from (44).
\end{proof}
\begin{lemma}
Let $x_1, x_2, \ldots, x_n$~ be independent variables over the
rational numbers field~$\mathbb{Q},$ $s_0=1, $
$$
s_k=s_k(x_1, x_2, \ldots, x_n)=\sum_{1 \leqslant i_1 <i_2< \ldots <
i_n \leqslant n} x_{i_1} x_{i_2} \ldots x_{i_k}, \quad 1 \leqslant k
\leqslant n,
$$
$ s_l=s_l(x_1, x_2, \ldots, x_n)=0$, if $l\geqslant n+1.$ Then for
all $k \geqslant 1$
\begin{equation}
s_k=(-1)^k \sum_{\substack{\lam_1+2\lam_2+ \ldots + k \lam_k =
k \\ \lambda_i \in \mathbb{Z}, \, \lambda_i \geqslant 0, \,
1 \leqslant i \leqslant k }}\frac{(-1)^{\lam_1 + \lam_2 + \ldots + \lam_k}} {1^{\lam_1}
2^{\lam_2} \ldots k^{\lam_k} \lam_1! \lam_2! \ldots \lam_k!}
\prod_{i=1}^{k}(x_1^i+ x_2^i + \ldots + x_n^i )^{\lam_i}
\end{equation}
\end{lemma}

\begin{proof}
Since
$$
\prod_{i=1}^n(x-x_i)=x^n+ \sum_{k=1}^n (-1)^k s_k(x_{i_1} x_{i_2}
\ldots x_{i_n})x^{n-k},
$$
then (47) follows directly from (44) and (45).
\end{proof}

\vspace{2em}

{\Large\bf References}\\
$[1]$ R.~F.~Scott. Mathematical notes, Messenger of Math.10:
\,pp.142-149 (1881) \\
$[2]$ H.~Minc. On a Conjecture of R.\,F\,Scott(1881), Linear
Algebra and its Applications 28(1979), pp.141-153.\\
$[3]$ R.~Kittappa. Proof of a Conjecture of 1881 on Permanents,
Linear and Multilinear Algebra, 1981, vol.~10, pp.75-82.\\
$[4]$ H.~Minc. Permanents, Addison-Wesley, Reading, 1978.\\
$[5]$ P.~Doubilet. On the Foundations of Combinational Theory. VII:
Symmetric Functions through the Theory of Distribution and
Occupancy, Studies in Applied
Mathematics, vol.1,~No.~4, 1972, pp.377 --- 396.\\
$[6]$ P.~Doubilet, G.-C.~Rota, R.~Stanley. On the foundations of
combinatorial theory~(VI): The idea of generating function, Proc.
Sixth Berkeley Symp. Math.
Stat. and Prob., vol.2, University of California Press,~1973,~pp.267 --- 318. \\
$[7]$ G.-C.~Rota. On the foundations of combinatorial theory. I.
Theory of M\"{o}bius functions, Z.~Wahrscheinlichkeitstheorie und
Verw.
Gebiete,~2,~1964,~pp.340 --- 368. \\
$[8]$ E.~A.~Bender, J.~R.~Goldman. On the applications of M\"{o}bius
inversion in
combinatorial analysis. The American Math Monthly,~82,~No.~8(1975),~pp.789 --- 803.\\
$[9]$ Waring, Meditationes algebraicae, Cantabrigiae 1670, 1(15,16).   \\
$[10]$ A.~M.~Kamenetsky. Permanents and determinants of group
matrices. Explicit formulas and recursion relations for permanents
and determinants of circulants. Solution of Lehmer's problem
evaluation coefficients of determinants of general circulants
// Deposited at VINITI 14.08.1990, No. 4620~---~V90, 499p. (in Russian)
\end{document}